\documentclass[final]{siamart1116}

\usepackage{graphicx,color}
\usepackage{amsmath,amssymb}
\usepackage{enumerate}
\usepackage{array} 
\usepackage[caption=false]{subfig} 
\makeatletter
     \@addtoreset{equation}{section}
     \def\theequation{\thesection.\arabic{equation}}
     \@addtoreset{figure}{section}
     \def\thefigure{\thesection.\@arabic\c@figure}
     \@addtoreset{table}{section}
     \def\thetable{\thesection.\@arabic\c@table}
     \makeatother
  \theoremstyle{plain}
  \theoremheaderfont{\normalfont\itshape}
  \theorembodyfont{\normalfont}
  \theoremseparator{.}
  \theoremsymbol{}
  \newtheorem{eg}{Example}
\newsiamremark{remark}{Remark}

\crefname{eg}{Example}{Examples} 
\crefname{code}{line}{lines} 
\crefname{mylineno}{line}{lines} 
\usepackage{upref} 

\graphicspath{{figs/}}  

\def\eu{\ensuremath{\mathrm{e}}}
\def\iu{\ensuremath{\mathrm{i}}}

\makeatletter
\def\mymatrix#1{\null\,\vcenter{\normalbaselines\m@th
   \ialign{\hfil$##$\hfil&&\quad\hfil$##$\hfil\crcr
     \mathstrut\crcr\noalign{\kern-\baselineskip}
     #1\crcr\mathstrut\crcr\noalign{\kern-\baselineskip}}}\,}
\makeatother
\def\mybmatrix#1{\left[\mymatrix{#1}\right]}

\def\mmax{m_{\max}}
\def\pmax{p_{\max}}

\newcommand{\C}{\mathbb{C}}
\newcommand{\R}{\mathbb{R}}
\newcommand{\nbyn}{{n \times n}}

\newcommand{\nbyq}{{n \times q}}
\newcommand{\nbynz}{{n \times n_0}}

\DeclareMathOperator{\diag}{diag}

\def\Re{\mathop{\mathrm{Re}}}
\def\Im{\mathop{\mathrm{Im}}}

\def\At{\widetilde{A}}

\def\a{\alpha}
\def\b{\beta}
\def\th{\theta}
\def\thhat{\widehat{\theta}}
\def\xhat{\widehat{x}}

\def\Alg{Algorithm}
\def\alg{algorithm}

\def\norm#1{\|#1\|}
\def\normF#1{\|#1\|_F}
\def\normi#1{\|#1\|_1}

\def\t#1{\texttt{#1}}
\def\resp{respectively}
\def\py{polynomial}
\def\mmax{m_{\max}}
\def\Cmstar{C_{m_*}}
\def\rom#1{{\upshape#1}}

\def\r{\mathrm{r}}
\def\i{\mathrm{i}}
\def\cn{condition number}

\def\rom#1{{\upshape#1}}
\newcommand{\parens}[1]{\rom{(}#1\rom{)}}

\newcounter{mylineno}
\makeatletter
\let\oldtabcr\@tabcr

\def\mynewline{\refstepcounter{mylineno}%
               \llap{\footnotesize\arabic{mylineno}\hspace{5pt}}%
              }

\gdef\@tabcr{\@stopline \@ifstar{\penalty%
           \@M \@xtabcr}\@xtabcr\mynewline}

\newenvironment{code}{%
                        \mathcode`\:="603A  
                        
                        \setcounter{mylineno}{0}
                        \par
                        \upshape
                        \begin{list} 
                           {} {\leftmargin = 1cm}
                        \item[]
                        \begin{tabbing}

                           \hspace*{.3in} \= \hspace*{.3in} \=
                           \hspace*{.3in} \= \hspace*{.3in} \= \kill
                           \mynewline
                       }{\end{tabbing}\end{list}}
\makeatother
\newcommand{\tolh}{u_{\mathrm{half}}}
\newcommand{\tols}{u_{\mathrm{single}}}
\newcommand{\told}{u_{\mathrm{double}}}

\newcommand{\tol}{\mathrm{tol}}

\newcommand{\rme}{\mathrm{e}}

\newcommand{\argmin}{\mathop{\operatorname{arg\,min}}}

\newcommand{\trigmv}{\texttt{trigmv}}
\newcommand{\trighmv}{\texttt{trighmv}}
\newcommand{\expmvt}{\texttt{expmv\_tspan}}
\newcommand{\trigexpmv}{\texttt{trig\_expmv}}
\newcommand{\trighexpmv}{\texttt{trigh\_expmv}}
\newcommand{\expmv}{\texttt{expmv}}
\newcommand{\expmvtspan}{\texttt{expmv\_tspan}}
\newcommand{\cosm}{\texttt{cosm}}
\newcommand{\sinm}{\texttt{sinm}}
\newcommand{\dense}{\texttt{dense}}
\newcommand{\expleja}{\texttt{expleja}}
\newcommand{\trigblock}{\texttt{trig\_block}}
\newcommand{\trighblock}{\texttt{trigh\_block}}

\begin{document}
\title{Computing the Action of Trigonometric and Hyperbolic Matrix Functions%
       \thanks{Version of January 1, 2017.}}
\author{Nicholas J. Higham\thanks{%
  School of Mathematics, The University of Manchester, Manchester, M13 9PL, UK.
       Supported by European Research Council
       Advanced Grant MATFUN (267526)
       and Engineering and Physical Sciences Research Council
       grant EP/I01912X/1.} \and 
  Peter Kandolf\thanks{Institut f\"ur Mathematik, Universit\"at Innsbruck, 
  Austria, Supported by a DOC Fellowship of the Austrian
  Academy of Science at the Department of Mathematics, University of Innsbruck,
  Austria.}}

\newcommand{\peter}[2][]{\todo[color=white, #1]{#2}} 
\newcommand{\nick}[2][]{\todo[color=blue!40!white, #1]{#2}}

\maketitle

\begin{abstract}
We derive a new algorithm for computing the action $f(A)V$ of the cosine,
sine, hyperbolic cosine, and
hyperbolic sine of a matrix $A$ on a matrix $V$,
without first computing $f(A)$.
The algorithm can compute 
$\cos(A)V$ and $\sin(A)V$ simultaneously, and likewise for 
$\cosh(A)V$ and $\sinh(A)V$,
and it uses only real arithmetic when $A$ is real.
The algorithm exploits an existing algorithm \texttt{expmv}
of Al-Mohy and Higham for $\mathrm{e}^AV$ 
and its underlying backward error analysis.
Our experiments show that the new algorithm performs in
a forward stable manner and is generally significantly faster than
alternatives based on multiple invocations of \texttt{expmv}
through formulas such as 
$\cos(A)V = (\mathrm{e}^{\mathrm{i}A}V + \mathrm{e}^{\mathrm{-i}A}V)/2$.
\end{abstract}



\section{Introduction}\label{sec:int}

This work is concerned with the computation of $f(A)V$
for trigonometric and hyperbolic functions $f$,
where $A\in\C^{\nbyn}$ and
      $V\in\C^{\nbynz}$ with $n_0 \ll n$.
Specifically, we consider the computation of
the actions of the matrix cosine, sine, 
hyperbolic cosine, and hyperbolic sine functions.
\Alg s exist for computing these matrix functions,
such as those in \cite{ahr15}, \cite{hahi05c},
but we are not aware of any existing \alg s for 
computing their actions.

Applications where these actions are required include differential
equations (as discussed below) and network analysis
\cite{ehh08}, \cite{kgg12}.
Furthermore, the proposed algorithm can also be utilized to compute the action 
of the matrix exponential or $\varphi$ functions at different time steps. 
This, in return, finds an application in the efficient implementation of 
exponential integrators \cite{ho10}. 
One distinctive feature of the algorithm proposed is that it avoids complex 
arithmetic for a real matrix. 
This characteristic can be exploited to use only real arithmetic in the 
computation of the matrix exponential as well, 
if the matrix is real but the step argument complex.
This is useful for higher order splitting methods
\cite{aeoa09}, or for the solution of the Schr\"odinger equation, 
where the problem can be rewritten so that the step argument is complex 
and the matrix is real (see~\cref{eg:schroedinger}).

One line of attack is to develop \alg s for $f(A)V$ for each of these
four $f$ individually.
An \alg\ \expmv\ of Al-Mohy and Higham \cite{alhi11} 
for computing the action of the
matrix exponential relies on the scaling and powering relation
$\eu^A b = (\eu^{A/s})^sb$, for nonnegative integers $s$, and uses a Taylor
\py\ approximation to $\eu^{A/s}$.
The trigonometric functions $\cos$ and $\sin$
do not enjoy the same relation,
and while the double- and triple-angle formulas
$\cos (2A) = 2\cos^2(A) - I$ and
$\sin (3A) = 3\sin (A) - 4\sin^3 (A)$
can be successfully used in computing the cosine and sine~\cite{ahr15},
they do not lend themselves to computing the action of these functions.
For this reason our focus will be on exploiting the \alg\ of \cite{alhi11}
for the action of the matrix exponential.
While this approach may not be optimal for each of the four $f$,
we will show that it leads to a numerically reliable \alg\
and has the advantage that it allows the use of existing software.

The matrix cosine and sine functions arise in solving the
system of second order differential equations
\begin{align*}
\frac{\mathrm{d}^2}{\mathrm{d}t^2}y + A^2 y = 0, \qquad y(0)=y_0,\quad y'(0)=y_0',
\end{align*}
whose solution is given by
\begin{align*}
    y(t) = \cos(tA)y_0 + A^{-1}\sin(tA)y_0'.
\end{align*}
Note that $A^2$ is the given matrix,
so $A$ may not always be known or easy to obtain.
By rewriting this system as a first order system of twice the dimension the
solution can alternatively be obtained as the first component of the action
of the matrix exponential:
\begin{align}\label{exp2by2}
  \begin{bmatrix}
  y(t) \\ y(t)'
  \end{bmatrix}
  &=
\exp\left(
  t\begin{bmatrix}
  0 & I\\-A^2 & 0
  \end{bmatrix}
  \right)
  \begin{bmatrix}
  y_0\\y_0'
  \end{bmatrix}
  =
  \begin{bmatrix}
  \cos(tA) & A^{-1}\sin(tA)\\-A\sin(tA)&\cos(tA)
  \end{bmatrix}
  \begin{bmatrix}
  y_0\\y_0'
  \end{bmatrix}\\
  &=
  \begin{bmatrix}
  \cos(tA)y_0 + A^{-1}\sin(tA)y_0'\\-A\sin(tA)y_0 + \cos(tA)y_0'
  \end{bmatrix}. \nonumber
\end{align}
By setting $y_0 = b$ and $y_0' = 0$,
or         $y_0 = 0$ and $y_0' = b$,
and solving a linear system with $A$ or multiplying by $A$,
\resp, we obtain $\cos(tA)b$ and $\sin(tA)b$.
However, as a general purpose algorithm,
making use of \expmv\ from \cite{alhi11},
this approach has several disadvantages.
First, each step requires two matrix--vector products with $A$,
when we would hope for one.
Second, because the block matrix has zero trace, no shift is applied by
\expmv, so an opportunity is lost to reduce the norms.
Third, the coefficient matrix is nonnormal (unless $A^2$ is orthogonal),
which can lead to higher computational cost \cite{alhi11}.

We recall that all four of the functions addressed here can be expressed
as linear combinations of exponentials~\cite[chap.~12]{high:FM}:
\begin{subequations}\label{eq:both_trig_funs}
\begin{alignat}{2}\label{eq:trigh_funs}
 \cosh A  &= \tfrac{1}{2}(\eu^{A}+ \eu^{-A}), &\qquad
 \sinh A  &= \tfrac{1}{2}(\eu^{A}- \eu^{-A}), \\
\label{eq:trig_funs}
 \cos A  &= \tfrac{1}{2}(\eu^{\iu A}+ \eu^{-\iu A}),  &\qquad
 \sin A  &= \tfrac{-\iu}{2}(\eu^{\iu A}- \eu^{-\iu A}).
\end{alignat}
\end{subequations}
Furthermore, we have
\begin{align}\label{eq:trig_real}
  \eu^{\iu A} &= \cos A + \iu\sin A,
\end{align}
which implies that for real $A$,
$\cos A = \Re \eu^{\iu A}$ and
$\sin A = \Im \eu^{\iu A}$.
The main idea of this paper is to exploit these formulas to compute
$\cos(A) V$, $\sin(A) V$,
$\cosh(A) V$, and $\sinh(A) V$
by computing
$\eu^{\beta A}V$ and
$\eu^{-\beta A}V$ simultaneously with $\beta = \iu$ and $\beta = 1$,
using a modification of the \alg\ \expmv\ of \cite{alhi11}.


In \cref{sec:ba} we discuss the backward error of the underlying
computation.
In \cref{sec:alg} we present the algorithm and the computational aspects.
Numerical experiments are given in \cref{sec:ne},
and in \cref{sec.conc} we offer some concluding remarks.

\section{Backward error analysis}\label{sec:ba}

The aim of this section is to bound the backward error for
the approximation of $f(A)V$ using truncated Taylor series expansions of
the exponential,
for the four functions $f$ in~\cref{eq:both_trig_funs}.
Here, backward error is with respect to truncation errors in the
approximation, and exact computation is assumed.

We will use the analysis of Al-Mohy and Higham \cite{alhi11},
with refinements to reflect the presence of two related exponentials
in each of the definitions of our four functions.

It suffices to consider the approximation of $\eu^{A}$,
since the results apply immediately to $\eu^AV$.
We consider a general approximation $r(A)$,
where $r$ is a rational function,
since when $r$ is a truncated Taylor series no simplifications accrue.

Since $A$ appears as $\pm A$ and $\pm \iu A$ in \cref{eq:both_trig_funs},
in order to cover all cases
we treat $\b A$, where $|\b| \le 1$.
Consider the matrix
$$
      G = \eu^{-\b A} r(\b A) - I.
$$
With $\log$ denoting the principal matrix logarithm \cite[sec.~1.7]{high:FM},
let
\begin{equation}\label{Eeq}
  E = \log( \eu^{-\b A} r(\b A)) = \log (I + G),
\end{equation}
where $\rho(G) < 1$ is assumed for the existence of the logarithm.
We assume that $r$ has the property that
$r(X) \to \eu^X$ as $X \to 0$,
which is enough to ensure that $\rho(G) < 1$ for small enough $\b A$.

Exponentiating \cref{Eeq},
and using the fact that all terms commute
(each being a function of $A$), we obtain
$$
     r(\b A) = \eu^{ \b A + E},
$$
so that $E$ is the backward error matrix for the approximation.

For some positive integer $\ell$ and some radius of convergence $d > 0$
we have, from \cref{Eeq}, the convergent power series expansion
$$
      E = \sum_{i=\ell}^{\infty} c_i (\b A)^i, \qquad |\b|\rho(A) < d.
$$
We can bound $E$ by taking norms to obtain
\begin{equation}\label{Ebound}
   \norm{E} \le
    \sum_{i=\ell}^{\infty} |c_i| \, \norm{\b A}^i =: g( \norm{\b A} ).
\end{equation}
Assuming that $g(\th) = O(\th^2)$, the quantity
\begin{align}\label{eq:thhat}
   \thhat :=\max\{\,\th >0 : \, \th^{-1}g(\th) \le \tol\,\}
\end{align}
exists and we have the backward error result that
$\norm{\b A} \le \thhat$ implies
$r(\b A) = \eu^{\b A + E}$, with $\norm{E}\le \tol\, \norm{\b A}$.
Here $\tol$ represents the tolerance specified for the backward error.

In practice, we use scaling to achieve the required bound on $\norm{\b A}$,
so our approximation is $r(\b A/s)^s$ for some nonnegative integer $s$.
With $s$ chosen so that $\norm{\b A/s} \le \thhat$, we have
$$
  r(\b A/s)^s = \eu^{\b A + sE}, \qquad
   \frac{\norm{sE}}{\norm{\b A}} \le \tol.
$$

The crucial point is that since
$g(\norm{\b A} ) = g(|\b| \norm{A}) \le g(\norm{A})$, for all $|\b| \le 1$,
the parameter $s$ chosen for $A$ can be used for $\b A$.
Consequently, the original analysis gives the same bounds for
$\pm A$ and $\pm \iu A$ and the same parameters can be used for the computation
of all four of these functions.
This result does not state that the backward error is the same for each $\b$,
but rather the weaker result that each of the backward errors
satisfies the same inequality.


In practice, we use in place of
$\norm{\b A}$ in \cref{Ebound} the quantity $\a_p(\b A)$, where
\begin{equation}\label{eq:alphap}
   \a_p(X) = \max(d_p,d_{p+1}), \quad d_p = \norm{X^p}^{1/p},
\end{equation}
for some $p$ with $\ell \ge p(p-1)$,
which gives potentially much sharper bounds, as shown in
\cite[Thm.\ 4.2(a)]{alhi09a}.



Our conclusion is that all four matrix functions
appearing in \cref{eq:both_trig_funs}
can be computed in a backward stable manner with the same parameters.
As we will see in the next section, the computations can even be combined to
compute the necessary values simultaneously.

\section{The basic algorithm}\label{sec:alg}

As our core \alg\ for computing the action of the matrix exponential we take
the truncated Taylor series \alg\ of Al-Mohy and Higham~\cite{alhi11}. 
We recall some details of the algorithm below.
Other \alg s, such as the Leja method presented in~\cite{ckor15},
can be employed in a similar fashion, though the details will be different.
The truncated Taylor series \alg\ takes
\begin{align*}
r(A) \equiv T_m(A)=\sum_{j=0}^m\frac{A^j}{j!}.
\end{align*}
As suggested in \cite{alhi11}, we limit the degree 
$m$ of the polynomial approximant $T_m$ to $\mmax = 55$.
In order to allow the algorithm to work for general matrices, 
with no restriction on the norm, 
we introduce a scaling factor $s$ and assume that $\eu^{s^{-1}A}$ 
is well approximated by $T_m(s^{-1}A)$. 
From the functional equation of the exponential we have
\begin{align*}
  \eu^{A}V = \left(\eu^{s^{-1}A}\right)^sV,
\end{align*}
and so the recurrence
\begin{align*}
    V_{i+1} = T_m(s^{-1}A)V_i, \quad i=0, \ldots, s-1,\quad V_0=V,
\end{align*}
yields the approximation $V_s\approx\eu^{A}V$.
For a given $m$,
the function $g$ in \cref{Ebound} has 
$\ell = m+1$.


The parameter $\thhat$ in~\cref{eq:thhat}, which we now denote by $\th_m$,
depends on the \py\ degree $m = \ell-1$ and the tolerance $\tol$, and its
values are given in \cite[Table~3.1]{alhi09a} for IEEE single precision
arithmetic and double precision arithmetic.
The cost function
\begin{align}
  C_m(A) = ms = m\max\left\{1,\left\lceil\alpha_p(A)/\th_m\right\rceil\right\}
\end{align}
measures the number of matrix--vector products,
and the optimal degree $m_*$ is chosen in \cite{alhi11} such that
\begin{align}\label{eq:minimalcost}
C_{m_*}(A)=\min\left\{
  m\left\lceil\alpha_p(A)/\th_m\right\rceil\!:\ 2\leq p\leq \pmax,
  \; p(p-1)-1\leq m \leq \mmax
  \right\}.
\end{align}
Here, $\mmax$ is the maximal admissible Taylor \py\ degree
and $\pmax$ is the maximum value of $p$ such that $p(p-1)\le \mmax + 1$, 
to allow the use of \cref{eq:alphap}. 
Furthermore, $\pmax=8$ is the default choice in the implementation.
The parameters $m_*$ and $s$ are determined by \cref{cf:ms},
which is \cite[Code Fragment 3.1]{alhi11}.

\begin{algorithm}
\caption{$[m_*,s]=\text{parameters}(A,B,\tol)$\label{cf:ms}}
This code determines $m_*$ and $s$ given 
$A\in\C^{\nbyn}$, $B\in\C^{\nbyq}$, 
$\tol$, $\mmax$, and $\pmax$.
It is assumed that the $\a_p$ in
\cref{eq:minimalcost} are estimated using the 
block $1$-norm estimation \alg\ of Higham and Tisseur
\rom{\cite{hiti00n}} with two columns. 

\begin{code}
if
   $\normi{A} \le \dfrac{4}{q} \dfrac{\th_{\mmax}}{\mmax}
                            p_{\max}(p_{\max}+3)$ \\
\>$m_* = \argmin_{1\le m\le \mmax} m \lceil \normi{A}/\th_m \rceil$\\
\>$s = \lceil \normi{A}/\th_{m_*} \rceil$\\
else\\
\> Let $m_*$ be the smallest $m$ achieving the minimum in \cref{eq:minimalcost}.\\
\> $s = \max\bigl\{ \Cmstar(A)/m_*, 1\bigr\}$\\
end
\end{code}

\end{algorithm}

A further reduction of the cost can be 
achieved by choosing an appropriate point $\mu$ as the centre 
of the Taylor series expansion. 
As suggested in \cite{alhi11}, 
the shift is selected such that the Frobenius norm 
$\normF{A-\mu I}$ is minimized, that is,
$\mu=\operatorname{trace}(A)/n$.

\Alg~3.2 of \cite{alhi11} computes $\eu^AB = [\eu^A b_1,\dots,\eu^Ab_q]$,
that is, the action of $\eu^{A}$ on several vectors.
The following modification of that \alg\ essentially computes
$[\eu^{\tau_1A} b_1,\dots,\eu^{\tau_qA}b_q]$:
the actions at different $t$ values.
The main difference between our \alg\ and
\cite[Alg.~3.2]{alhi11} is in \cref{alg:recurence_line} of \cref{alg:basic_alg},
where a scalar ``$t$'' has been changed to a (block) diagonal matrix
$$
      D(\tau) = D(\tau_1,\tau_2,\dots,\tau_q) \in\C^{q \times q}
$$
that we define precisely below.
The exponential computed in \cref{alg:eta_line} of \cref{alg:basic_alg} 
is therefore a matrix exponential.

For simplicity we omit balancing, but it can be applied in the same way as
in \cite[Alg.~3.2]{alhi11}.

\begin{algorithm}\caption{$F = \textbf{F}(D(\tau),A,B)$}\label{alg:basic_alg}
For $\tau_1,\ldots,\tau_q\in\C$,
$A\in\C^{\nbyn}$, $B\in\C^{\nbyq}$, and
a tolerance $\tol$ the following algorithm produces a matrix
$F = g\circ g \circ \cdots \circ g(B) \in\C^{n \times q}$ 
\parens{$s$-fold composition}
where
$g(B) \approx \left(B + \frac{1}{s}\At BD(\tau) + \frac{1}{2!s^2}\At^2BD(\tau)^2 + 
            \frac{1}{3!s^3} \At^3BD(\tau)^3 + \cdots\right)J$,
where $\At$ and $J$ are given in the \alg.

\begin{code}
  $\At=A-\mu I$, where $\mu=\operatorname{trace}(A)/n$\\
  $t=\max_k|\tau_k|$\\
  if $t\|\At\|_1=0$\\
  \> $m_*=0$, $s=1$\\
  else\\
  \> $[m_*,s]=\text{parameters}(t\At,B,\tol)$ \% \cref{cf:ms}\\
  end\\\label{alg:eta_line}
  $F=B$, $J = \eu^{\mu D(\tau)/s}$\\
  for $k=1:s$\\
  \> $c_1 = \|B\|_\infty$\\\label{alg:innerloop_start}
  \> for $j=1:m_*$\\\label{alg:recurence_line}
  \> \> $B=\At B(D(\tau)/(sj)),$\\
  \> \> $c_2=\|B\|_\infty$\\
  \> \> $F = F+ B$\\
  \> \> if $c_1 + c_2 \leq \tol \|F\|_\infty$, break, end\\
  \> \> $c_1=c_2$\\\label{alg:innerloop_end}
  \> end\\\label{alg:Dmult}
  \> $F = F J$, $B = F$\\
  end
\end{code}
\end{algorithm}

Note that for $\At = A - \mu I$,
$B = [b_1, b_2]$ and 
$D(\tau) = \diag(\tau_1,\tau_2)$ we have $g(B) = [g_1, g_2]$
in \cref{alg:basic_alg}, with 
\begin{align*}
   g_j &\approx \left(b_j + \frac{\At}{s} b_j \tau_j
                   + \frac{\At^2}{s^22!} b_j \tau_j^2
                   + \frac{\At^3}{s^33!} b_j \tau_j^3 + \cdots\right)
                   \rme^{\mu\tau_j/s}\\
        &= \eu^{(A-\mu I)\tau_j/s} \rme^{\mu\tau_j/s}b_j=
        \eu^{A\tau_j/s}b_j,
\end{align*}
for $j = 1,2$.
Therefore we can compute the four actions of interest by selecting
appropriately
$\tau_1$, $\tau_2$, and $B$ and carrying out some postprocessing.
For given $t$, $A$, and $b$ we can compute, 
with $\textbf{F}$ as in \cref{alg:basic_alg},
\begin{enumerate}
\item an approximation of $\cosh(tA)b$ by
\begin{align*}
	B=[b/2,b/2], \quad D(\tau)=\begin{bmatrix}t& 0\\0 & -t\end{bmatrix},\quad
  \cosh(tA)b=\textbf{F}(D(\tau),A,B)\begin{bmatrix}1\\1\end{bmatrix};
\end{align*}
\item an approximation of $\sinh(tA)b$ by
\begin{align*}
	B=[b/2,b/2], \quad D(\tau)=\begin{bmatrix}t& 0\\0 & -t\end{bmatrix},
  \quad 
  \sinh(tA)b=\textbf{F}(D(\tau),A,B)\begin{bmatrix}1\\-1\end{bmatrix};
\end{align*}
\item an approximation of $\cos(tA)b$ by
\begin{align*}
	B=[b/2,b/2], \quad D(\tau)=\begin{bmatrix}\iu t& 0\\0 & -\iu t\end{bmatrix},
  \quad 
  \cos(tA)b=\textbf{F}(D(\tau),A,B)\begin{bmatrix}1\\1\end{bmatrix};
\end{align*}
\item an approximation of $\sin(tA)b$ by
\begin{align*}
	B=[b/2,b/2], \quad D(\tau)=\begin{bmatrix}\iu t& 0\\0 & -\iu t\end{bmatrix},
  \quad 
  \sin(tA)b=\textbf{F}(D(\tau),A,B)\begin{bmatrix}-\iu\\\iu\end{bmatrix}.
\end{align*}
\end{enumerate}


Obviously, since they share the same $B$ and $D(\tau)$,
we can combine the computation of
$\cosh(tA)b$ and $\sinh(tA)b$,
and $\cos(tA)b$ and $\sin(tA)b$, respectively,
without any additional cost.
Furthermore, it is also possible to combine the computation of all four matrix
functions by a single call to $\textbf{F}(D(\tau),A,B)$ with
$B = [b,b,b,b]/2$ and $D(\tau)=\diag[t,-t,\iu t, -\iu t]$.

If $A$ is a real matrix the computation of $\cos(tA)b$ and $\sin(tA)b$ can be
performed entirely in real arithmetic, as we now show.
We need the formula
\begin{equation}\label{exp2by2_small}
   \exp\left(
     \mybmatrix{0 & t \cr
               -t & 0 \cr} \right)
    = \mybmatrix{ \cos t & \sin t\cr
                 -\sin t & \cos t\cr}.
\end{equation}


\begin{lemma}\label{lem:real_case}
For $A\in\R^{\nbyn}$, $b=b_\r + \iu b_\i \in\C^{n}$, and
$t\in\R$, the vector
$f = f_\r+\iu f_\i = \text{\rm\bf F}(D(\iu t),A,b) \approx \eu^{\iu tA}b$ 
can be computed
in real arithmetic by
\begin{equation}\label{fcomplex}
\begin{bmatrix}f_\r,&f_\i\end{bmatrix}=
  \text{\rm\bf F}\!\left(D(\tau), A,
              \begin{bmatrix} b_\r, b_\i\end{bmatrix}\right),
  \quad\mbox{where} \quad\tau=\iu t, \quad
  D(\tau) = \begin{bmatrix} 0 & t\\-t & 0\end{bmatrix}.
\end{equation}
Furthermore, the resulting vectors $f_\r$ and $f_\i$ are approximations of, \resp,
\begin{align*}
  f_\r = \cos(tA)b_\r - \sin(tA)b_\i, \qquad
  f_\i = \sin(tA)b_\r + \cos(tA)b_\i.
\end{align*}
\end{lemma}

\begin{proof}
With $B = [b_\r,b_\i]$  we have 
\begin{align*}
   g(B) &\approx \left(\mybmatrix{b_\r ,b_\i\cr }
          + t\mybmatrix{ \dfrac{\At b_\r}{s} , \dfrac{\At b_\i}{s}\cr}
             \mybmatrix{ 0 & 1 \cr -1 & 0 \cr}
          + t^2\mybmatrix{ \dfrac{\At^2b_\r}{s^2 2!} 
                  , \dfrac{\At^2b_\i}{s^2 2!}\cr }\right.\\
        & \left.\qquad
          + t^3\mybmatrix{ \dfrac{\At^3b_\r}{s^3 3!} 
                  , \dfrac{\At^3b_\i}{s^3 3!}\cr }
             \mybmatrix{ 0 & 1 \cr -1 & 0 \cr}
          + \cdots\right)
          \exp\left(
          \mybmatrix{0 & t\mu/s \cr
                    -t\mu/s & 0 \cr} \right)
\end{align*}
and on collecting terms, applying \cref{exp2by2_small}
and the addition formulas \cite[Thm.~12.1]{high:FM},
and recalling that $\At = A - \mu I$, we find that 
\begin{align*}
  g(B) &\approx \mybmatrix{ \cos\left( \frac{t\At}{s} \right)b_\r
                   -\sin\left( \frac{t\At}{s} \right)b_\i, &
                    \sin\left( \frac{t\At}{s} \right)b_\r
                   +\cos\left( \frac{t\At}{s} \right)b_\i \cr}
                   \mybmatrix{ \cos \left(\frac{t\mu}{s}\right) 
                             & \sin \left(\frac{t\mu}{s}\right)\cr
                 -\sin \left(\frac{t\mu}{s}\right) 
                    & \cos \left(\frac{t\mu}{s}\right)\cr} \\
      &= \mybmatrix{ \cos\left( \frac{tA}{s} \right) & 
                     -\sin\left( \frac{tA}{s} \right) \cr
                     \sin\left( \frac{tA}{s} \right) & 
                     \cos\left( \frac{tA}{s} \right) \cr} 
         \mybmatrix{ b_\r \cr b_\i \cr}
       =:   C \mybmatrix{ b_\r \cr b_\i \cr}.
\end{align*}
Hence, overall, using \cref{exp2by2_small} again,
$$
   \text{\rm\bf F}(D(\iu t)A,b) \approx C^s \mybmatrix{ b_\r \cr b_\i \cr}
    = \mybmatrix{ \cos (tA) & -\sin (tA) \cr \sin (tA) & \cos (tA) \cr}
      \mybmatrix{ b_\r \cr b_\i \cr},
$$
as required.
\end{proof}

As a consequence of \cref{lem:real_case} we can compute,
with $D$ defined in \cref{fcomplex},
\begin{enumerate}
\item an approximation of $\cos(tA)b$ by
\begin{align*}
 B=[b,0], \quad \tau=\iu t,\quad
 D(\tau)=\begin{bmatrix}0& t\\-t & 0\end{bmatrix},
  \quad
  \cos(tA)b=\textbf{F}(D(\tau),A,B)\begin{bmatrix}1\\0\end{bmatrix};
\end{align*}
\item an approximation of $\sin(tA)b$ by
\begin{align*}
 B=[b,0], \quad \tau=\iu t,\quad
 D(\tau)=\begin{bmatrix}0& t\\-t & 0\end{bmatrix},
  \quad
  \sin(tA)b=\textbf{F}(D(\tau),A,B)\begin{bmatrix}0\\1\end{bmatrix}.
\end{align*}
\end{enumerate}
We compute the matrix exponential $J$ in \cref{alg:eta_line} of 
\cref{alg:basic_alg} 
by making use of \cref{exp2by2_small}.

We make three remarks.

\begin{remark}[Other cases.]
\Cref{alg:basic_alg} can also be used to compute exponentials at
different time steps and with the use of \cite[Thm.~2.1]{alhi11} it can be
used to compute linear combinations of $\varphi$ functions at different
time steps (see, e.g., \cite[sec.~10.7.4]{high:FM} for details of the
$\varphi$ functions).  This in turn is useful for the implementation of
exponential integrators~\cite{ho10}.  
The internal stages of an
exponential integrator often require the evaluation of a $\varphi$ function
at intermediate steps, e.g., $\varphi(c_k t A)b$ for $0<c_k\leq1$ and
$k\geq 1$.  Although the new algorithm can be used in these situations it
might not be optimal for each of the $c_k$ values as the parameters $m_*$
and $s$ are chosen for the largest value of $t$ and might not be optimal
for an intermediate point.  Nevertheless, the computation can be performed in
parallel for all the different values of $t$ and level-3 BLAS routines can
be used, which can speed up the process.
Furthermore, the algorithm could also be used to generate dense output,
in terms of the time step, as is sometimes desired for time integration.
\end{remark}

\begin{remark}\label{rem:expmvt}
We note that in \cite[Code Fragment 5.1, Alg.~5.2]{alhi11}
the authors also present an algorithm
to compute $\eu^{t_k A}b$ on equally spaced grid points $t_k=t_0+hk$ with
$h=(t_q-t_0)/q$.
With that code we can compute $\cosh(A)b$ and $\sinh(A)b$
by setting $t_0=-t$, $t_q=t$, and $q=1$, so that 
$b_1=\eu^{t_0A}b = \eu^{-tA}b$, $h=2t$, 
and $b_2=\eu^{hA}b_1 = \eu^{tA}b$.
This is not only slower than our approach, 
as the code now has to perform a larger time step and
compute the necessary steps consecutively and not in parallel,
but it can also cause instability.
In fact, for some of the matrices of \cref{eg:floatingpoint}
in \cref{sec:ne} we see a large error if we use 
\cite[Alg.~5.2]{alhi11} as outlined above.
Furthermore, as we compute with $\pm \beta$ we can optimize the algorithm by
using level-3 BLAS routines and we can avoid complex arithmetic by our direct
approach.
\end{remark}

\begin{remark}[Block version]\label{rem:block}
As indicated in the introduction it is sometimes required to compute
the action of our four functions
not on a vector but on a tall, thin matrix $V\in\C^{\nbynz}$.
It is possible to use \cref{alg:basic_alg} for this task.
One simply needs to repeat each $\tau_k$ value $n_0$ times and the matrix
$V$ needs to be repeated
$q$ times for each of the $\tau_k$ values
(this corresponds to replacing the vector $b$ by the matrix $V$ in the
definition of $B$).
This procedure can be formalized with the help of the Kronecker product
$X\otimes Y$.
We define the time matrix by $D(\tau)\otimes I_{n_0}$,
and the
postprocessing matrix $\tilde P$ by $P\otimes I_{n_0}$.
Furthermore, the matrix $B$ reads as $I_q\otimes V/2$.
For $V=[v_1, v_2]$ ($n_0=2$) the computation of $\cosh(tA)V$ becomes
\begin{align*}
  B=[v_1,v_2,v_1,v_2]/2, \quad
  \tilde{D}(\tau)=D(\tau)\otimes I_2 =
  \diag(t, t,-t, -t)
\end{align*}
and results in
\begin{align*}
\cosh(tA)V= \textbf{F}(\tilde{D}(\tau),A,B)\begin{bmatrix}I_2\\I_2\end{bmatrix}.
\end{align*}
\end{remark}

\section{Numerical experiments}\label{sec:ne}

Now we present some numerical experiments that
illustrate the behaviour of \cref{alg:basic_alg}.
All of the experiments were carried out in MATLAB R2015a (glnxa64) 
on a Linux machine
and for time measurements only one processor is used.
We work with three tolerances in \Alg~\ref{alg:basic_alg}, 
corresponding to half precision, single precision, and double precision,
\resp:
\begin{align*}
       \tolh &= 2^{-11} \approx 4.9 \times 10^{-4}, \\
       \tols &= 2^{-23} \approx 6.0 \times 10^{-8}, \\
       \told &= 2^{-53} \approx 1.1 \times 10^{-16}. 
\end{align*}
All computations are in IEEE double precision arithmetic.

We use the implementations of the \alg s of \cite{alhi11} from
\url{https://github.com/higham/expmv}, which are named
\expmv\ for  \cite[Alg.~3.2]{alhi11} and
\expmvtspan\ for \cite[Alg.~5.2]{alhi11}.
We also use the implementations \cosm\ and \sinm\ from
\url{https://github.com/sdrelton/cosm_sinm} of the \alg\ of 
\cite[Alg.~6.2]{ahr15} for computing the matrix sine and cosine;
the default option of using a Schur decomposition is chosen in the first
experiment, but no Schur decomposition is used in the second and third
experiments.
We note that we did not use the function \verb"cosmsinm" for a simultaneous 
computation as we found it less accurate than \cosm\ and \sinm\ in
\cref{eg:floatingpoint}. 

In order to compute $\cos(tA)b$ and $\sin(tA)b$ we use the
following methods.
\begin{enumerate}

\item
\trigmv{} denotes \cref{alg:basic_alg} with
  real or complex arithmetic (avoiding complex arithmetic when possible), 
  computing the two functions simultaneously.

\item
\trigexpmv{} denotes the use of \expmv, in two forms.
For a real matrix \expmv\ is called with
the pure imaginary step argument $\iu t$,
making use of \cref{eq:trig_real}.
For a complex matrix \expmv\ is called twice,
with step arguments $\iu t$ and $-\iu t$,
and \cref{eq:trig_funs} is used.

\item
\dense\ denotes the use of \cosm\ and \sinm\ to compute the dense matrices
$\cos(tA)$ and $\sin(tA)$ before the
multiplication with $b$.
\item \trigblock{} denotes the use of formula \cref{exp2by2} with $y_0=0$ and
$y'_0=b$. 
Therefore we need one extra matrix--vector product to compute $\sin(tA)b$. 
In order to compute the exponential we use \expmv{}. 
\end{enumerate}

For the computation of $\cosh(tA)b$ and $\sinh(tA)b$ we use the
following methods.
\begin{enumerate}

\item \trighmv{} denotes \cref{alg:basic_alg}, computing
  the two functions simultaneously.

\item \trighexpmv{} denotes the use of \expmv\
called twice
with $\pm t$ as step arguments.  

\item \expmvt{} denotes \cite[Alg.~5.2]{alhi11} called with
$t_0=-t$, $q=1$, and $t_q=t$, as discussed in of \cref{rem:expmvt}.

\item
\dense\ denotes the use of \cosm\ and \sinm\ to compute  
the dense matrices
$\cosh(tA)$ and $\sinh(tA)$
as $\cos(\iu t A)$ and $-\iu\sin(\iu t A)$, respectively,
before the multiplication with $b$.

\item \trighblock{} denotes the use of formula \cref{exp2by2} with $y_0=0$ and
$y'_0 = b$, where $\iu A$ is substituted for $A$.
We need one extra matrix--vector product to compute $\sinh(tA)b$. 
In order to compute the exponential we use \expmv{}. 
\end{enumerate}

All the methods except \dense\ support tolerances
$\tolh$, $\tols$, and $\told$, whereas \dense\ is designed to deliver
double precision accuracy.


In all cases, when \cref{cf:ms} is called to compute the optimal
scaling and truncation degree we use $m_{\mathrm{max}}=55$ and
$p_\mathrm{max}=8$.

We compute relative errors in the 1-norm, $\normi{x-\xhat}/\normi{x}$,
where $x = f(A)b$.
In \cref{eg:floatingpoint},
$\xhat$ denotes a reference
solution computed with the Multiprecision Computing Toolbox
\cite{adva-mct} at 100-digit precision.
In \cref{eg.large,eg:schroedinger}
the matrices are too large for multiprecision
computations so the reference solution $X$ is taken as that obtained via
$\cosm$ or $\sinm$.

\begin{eg}[Behavior for existing test sets]\label{eg:floatingpoint}
In this experiment we compare
\trigmv{}, \trigexpmv{}, and \dense.
We show only the results for $\cos$ and $\cosh$, 
as the results for $\sin$ and $\sinh$ are very similar.

As test matrices we use Set 1-3 from \cite[sec.~6]{alhi09a},
with dimensions $n$ up to $50$.
We remove all matrices from our test sets where any of the considered 
functions overflow;
the overflow also appears for the dense method considered and is due to 
the result being too large to represent.
The elements of the vector $b$ are drawn from the standard normal
distribution and are the same for each matrix.
We compare the algorithms for tolerances $\tolh$, $\tols$, and $\told$.


The relative errors are shown in \cref{fig:accuracy_cos},
with the test matrices ordered by decreasing \cn\ 
$\kappa_{\cos}$ of the matrix cosine.
The estimated condition number is computed by the \texttt{funm\_condest1}
function of the Matrix Function Toolbox~\cite{high-mft}.  The required
Fr\'{e}chet derivative is computed with the $2 \times 2$ block
form~\cite[sec.~3.2]{high:FM}.

\begin{figure}
\includegraphics{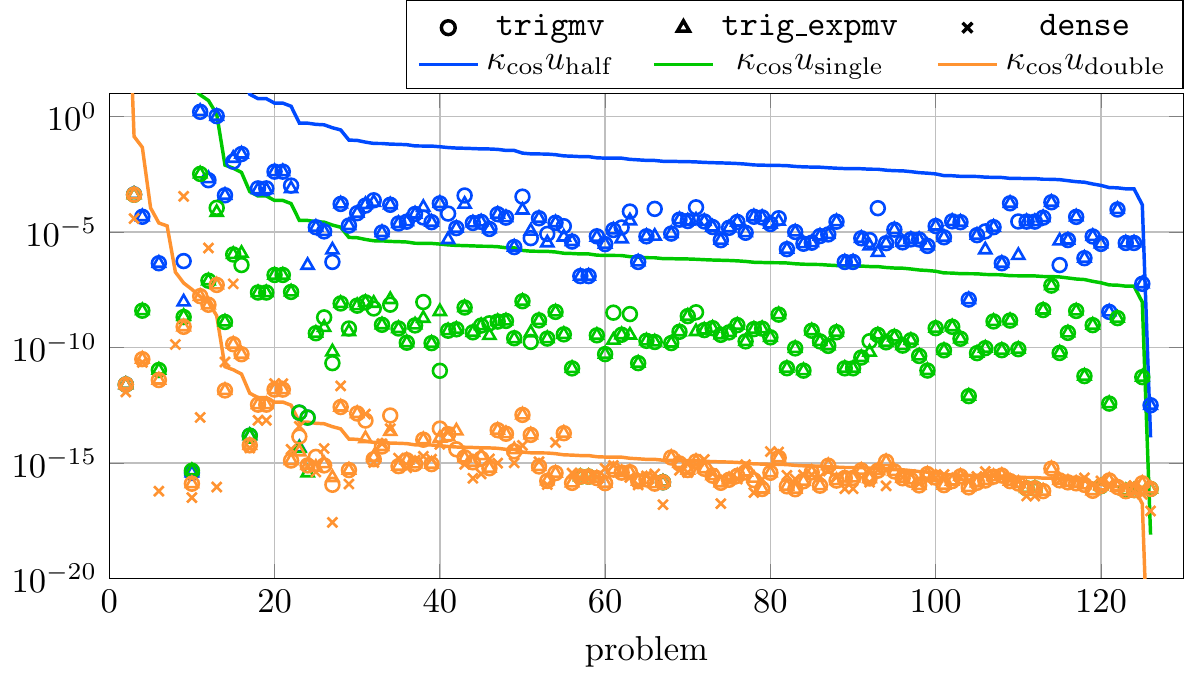}
\caption{\label{fig:accuracy_cos}Relative error in $1$-norm for computing
$\cos(A)b$ with three algorithms with tolerances
$\tolh$ (blue), $\tols$ (green), and $\told$ (orange).
The solid lines are the condition number multiplied by the tolerance.}

\end{figure}

\begin{figure}
\includegraphics{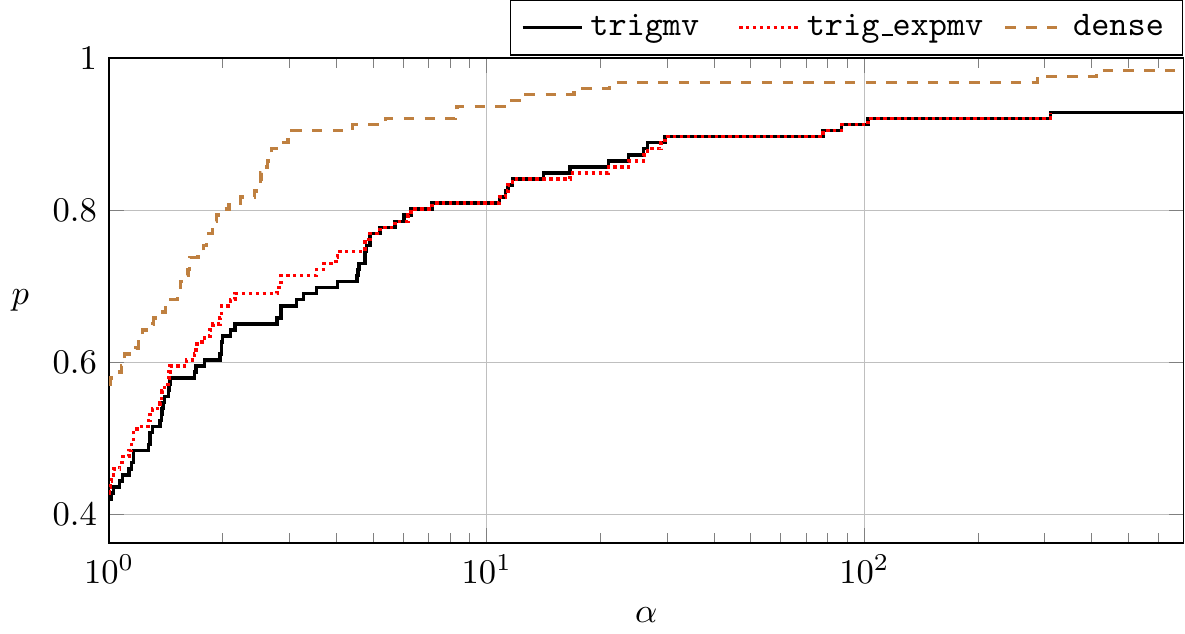}
\caption{\label{fig:performance_cos}
  Same data as in \cref{fig:accuracy_cos} for $\told$
  but presented as a performance profile.
  For each method, $p$ is the proportion of problems in which the error is 
  within a factor of $\alpha$ of the smallest error over all methods.
  }
\end{figure}



\begin{figure}
\includegraphics{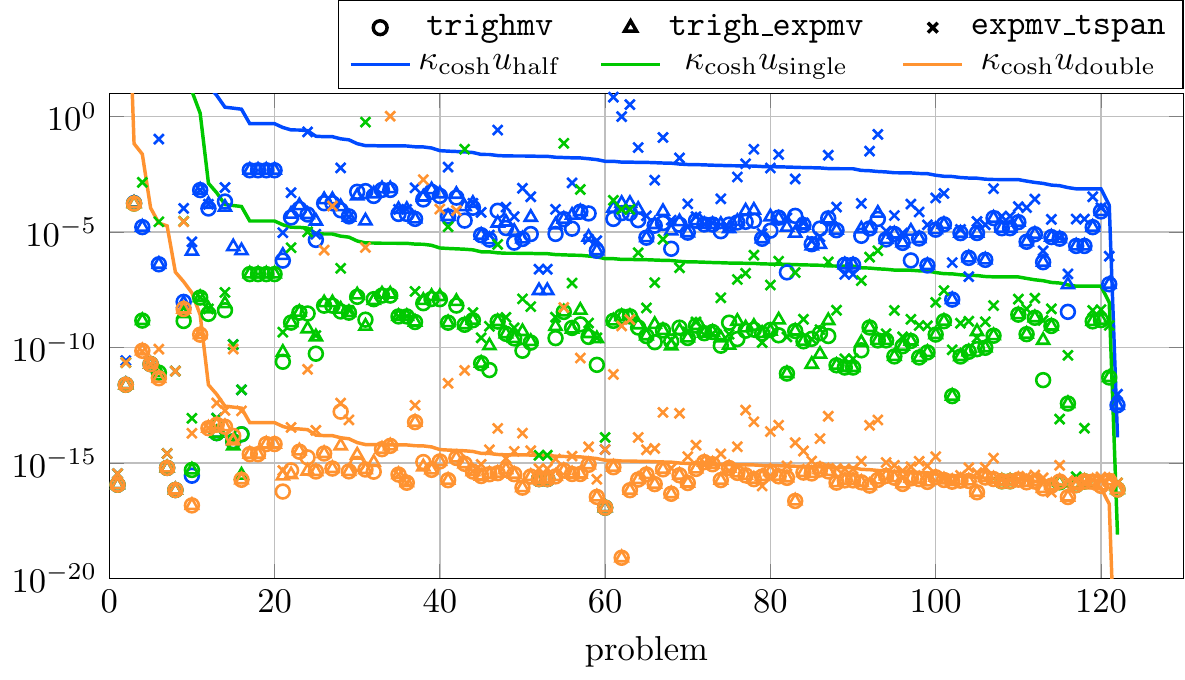}
\caption{\label{fig:accuracy_cosh}Relative error in $1$-norm for computing
$\cosh(A)b$ with tolerances
$\tolh$ (blue), $\tols$ (green), and $\told$ (orange).
The solid lines are the condition number multiplied by the tolerance.}
\end{figure}

\begin{figure}
\includegraphics{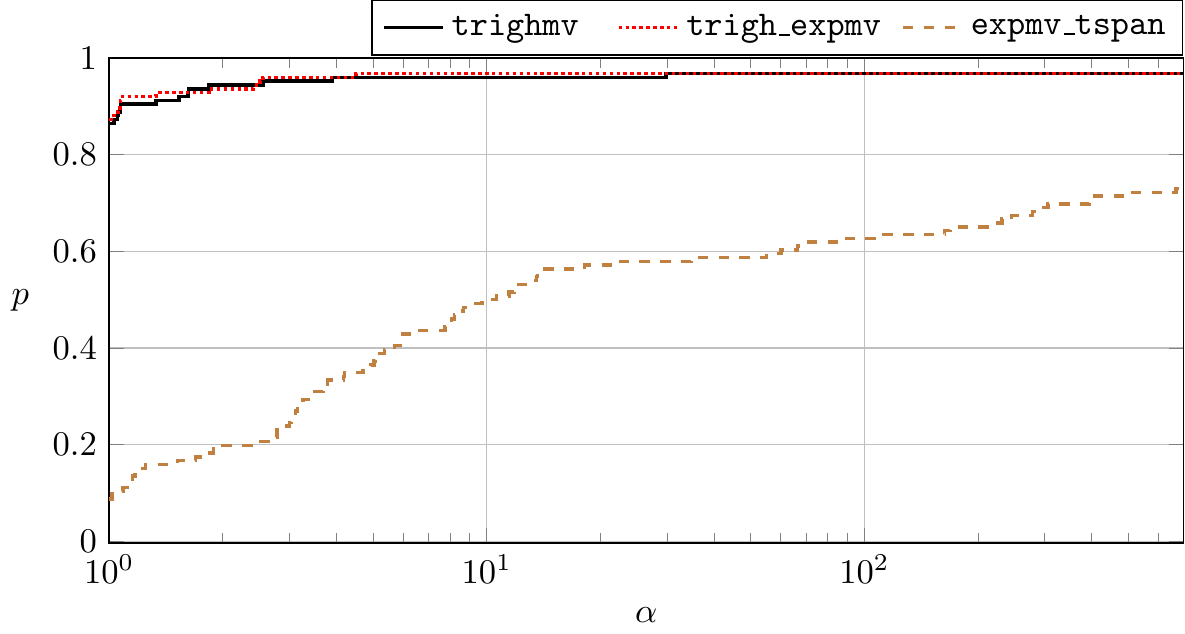}
\caption{\label{fig:performance_cosh}
Same data as in \cref{fig:accuracy_cosh}  for $\told$
  but presented as a performance profile.
  For each method, $p$ is the proportion of problems in which the error is 
  within a factor of $\alpha$ of the smallest error over all methods.
}
\end{figure}

From the error plot in \cref{fig:accuracy_cos} 
one can see that \trigmv{} and \trigexpmv\ behave in a forward stable manner,
that is, the relative error is always within a modest multiple of the \cn\
of the problem times the tolerance,
and likewise for \dense\ except for some mild instability on four problems.
We also show in \cref{fig:performance_cos} 
a performance profile for the experiment with tolerance $\told$.
In the performance profile
the curve for a given method shows 
the proportion of problems $p$ for which the 
error is within a factor $\alpha$ of the smallest error over all 
methods.
In particular, the value at $\alpha=1$ corresponds to the proportion
of problems 
where the method performs best and for large values of $\alpha$ the performance
profile gives an idea of the reliability of the method. 
The performance profile is computed with the code from
\cite[sec.~26.4]{hihi:MG3}
and we employ the idea of \cite{dihi13} to reduce the bias of relative
errors significantly less than the precision.
The performance profile suggests that the overall behavior of 
\trigmv\ and \trigexpmv\ 
is very similar.

For the computation of $\cosh$,
shown in \cref{fig:accuracy_cosh}, 
\expmvt{} is clearly not a good choice for the computation.
This is related to the implementation of \expmvt{}.
As the algorithm first computes $b_1=\eu^{-A}b$ and 
from this computes $b_2=\eu^{2A}b_1$ the result is not always stable,
as discussed in \cref{rem:expmvt}.
We see that \trighmv{} and \trighexpmv{} behave in a forward stable manner
and have about the same accuracy for all three tolerances, 
as is clear for double precision from the performance profile 
in \cref{fig:performance_cosh}.
\end{eg}

\begin{eg}[Behavior for large matrices]\label{eg.large}
In this experiment we take a closer look at the behavior of
several algorithms for large (sparse) matrices. 
For the computation of the trigonometric functions we compare
\trigmv{} with \trigblock{} and \trigexpmv{}, which both rely on \expmv{}.
For a real matrix \trigexpmv{} calls \expmv{} with a 
pure imaginary step argument and two calls are made for a complex matrix. 
For the hyperbolic functions, we compare \trighmv{} 
with \trighblock{} and \trighexpmv{}.
This time \trighexpmv{} always calls \expmv{} twice 
and \trighblock{} calls \expmv{} with a pure imaginary step argument. 
When \expmv{} is called several times the preprocessing step (\cref{cf:ms}) 
is only performed once.


We use the same matrices as in \cite[Example 9]{ckor15},
namely \texttt{orani676} and \texttt{bcspwr10}, which are obtained
from the University of Florida Sparse Matrix Collection \cite{dahu11}.
The matrix \texttt{orani676} is a nonsymmetric $2529\times 2529$ matrix 
with $90158$ nonzero entries and
\texttt{bcspwr10} is a symmetric $5300\times 5300$ matrix with 
$13571$ nonzero entries.
The matrix \texttt{triw} is \texttt{-gallery('triw',2000,4)}, 
which is a $2000\times 2000$ upper triangular matrix
with $-1$ in the main diagonal and $-4$ in the upper triangular part.
The matrix \texttt{triu} is an upper triangular matrix of dimension $2000$
with entries uniformly distributed on $[-0.5,0.5]$.
The $9801\times 9801$ matrix \texttt{L2} 
is from a finite difference discretization (second order symmetric differences)
of the two-dimensional Laplacian in the unit square.
The $27000\times 27000$ complex matrix \texttt{S3D}
is from  a finite difference discretization (second order symmetric differences)
of the
three-dimensional Schr\"odinger equation with harmonic potential in the unit cube,
The matrix \texttt{Trans1D} is a periodic, symmetric finite difference
discretization of the transport equation in the unit square with 
dimension $1000$.

As vector $b$ we use $[1,\ldots,1]^\mathrm{T}$ for \texttt{orani676},
$[1,0,\ldots,0,1]^\mathrm{T}$ for \texttt{bcspwr10},
the discretization of $256 x^2 (1-x)^2 y^2 (1-y)^2$
for \texttt{L2},
the discretization of $4096x^2(1-x)^2y^2(1-y)^2z^2(1-z)^2$
for \texttt{S3D},
the discretization of $\exp(-100(x-0.5)^2)$ for \texttt{Trans1D},
and $v_i=\cos\,i$ for all other examples.

\begin{table}\centering\captionsetup{position=top}\footnotesize
%
%
%
%
\caption{\label{tab:large matrices} Behavior of the algorithms for large
(sparse) matrices, for tolerance $\told$.}
\subfloat[\label{tab:large matrices_trig}Results for the computation of
$\cos$ and $\sin$.]{
  \centering\setlength{\tabcolsep}{4pt}
\begin{tabular}{rr|rl|rl|rl|l}
     &   & \multicolumn{2}{c|}{\trigmv{}}           & \multicolumn{2}{c|}{\trigexpmv{}} & \multicolumn{2}{c|}{\trigblock{}} & \dense\\
          & $t$ &  $mv$ &    Time & $mv$ & Time    & $mv$ & Time & Time\\\hline
\tt orani676 & 100 &2200    &  2.3e-1 & 4164    &  3.1e-1 & 2599    &  9.5e-1    &  2.8e2 \\
\tt bcspwr10 &  10 &618     &  4.1e-2 & 1500    &  1.2e-1 & 1392    &  1.2e-1 &  2.6e2 \\
\tt     triw &  10 &56740   &  5.7e1  & 113192  &  1.1e2  & 95389   &  1.2e2  &  1.9e1 \\
\tt     triu &  40 &3936    &  4.0    & 7524    &  8.5    & 4585    &  5.2    &  1.4e1 \\
\tt     L2   & 1/4 &107528  &  1.2e1  & 215320  &  1.9e1  & 257803  &  3.0e1  &  1.3e3 \\
 \hline
\end{tabular}}

%

\subfloat[\label{tab:large matrices_trigh}Results for the computation of
$\cosh$ and $\sinh$.]{
  \centering\setlength{\tabcolsep}{4pt}
\begin{tabular}{rr|rl|rl|rl|l}
     &   & \multicolumn{2}{c|}{\trighmv{}}           & \multicolumn{2}{c|}{\trighexpmv{}} & \multicolumn{2}{c|}{\trighblock{}}& \dense\\
          & $t$ &  $mv$ &    Time & $mv$   & Time    & $mv$ & Time &Time \\\hline
\tt orani676 & 100 &  2202  &  2.2e-1 & 2202    &  2.7e-1 &  2619  &  9.5e-1    &  2.1e2\\
\tt bcspwr10 &  10 &   632  &  4.1e-2 & 806     &  5.4e-2 &   855  &  7.4e-2 &  5.4e2\\
\tt triw     &  10 & 56478  &  5.7e1  & 57582   &  1.2e2  & 94499  &  1.1e2  &  1.2e2\\
\tt triu     &  40 &  4042  &  4.1    & 4031    &  8.0    &  4689  &  5.4    &  2.9e1\\
\tt  S3D     & 1/2 & 15962  &  1.7e1  & 15934   &  2.1e1  & 32135  &  1.4e1  &  2.3e4\\
\tt  Trans1D &   2 & 13086  &  2.0e-1 & 13039   &  2.4e-1 & 17551  &  2.1e-1 &  6.2  \\
 \hline
\end{tabular}
}
%
%
\end{table}

The results for computing $\cos(tA)b$ and $\sin(tA)b$ are shown in 
\Cref{tab:large matrices_trig}, and those for 
$\cosh(tA)b$ and $\sinh(tA)b$ in \Cref{tab:large matrices_trigh}.
The different algorithms are run with tolerance $\told$.
All the methods behave in a forward stable manner,
with one exception,
so we omit the errors in the table. 
The exception is the \trighblock{} method, which has an error about 
$10^2$ times larger than the other methods for \texttt{Trans1D}.
For the different methods we list the number of 
real matrix--vector products performed ($mv$),
as well as the overall time in seconds averaged over ten runs.
The tables also show the time the dense algorithm needed to compute the
reference solution (computing both functions simultaneously). 
%


In \cref{tab:large matrices_trig} we can see
that \trigmv{} always needs the fewest matrix--vector products and 
that with the sole exception of \texttt{triw} it is always the fastest method. 
We can also see that, as expected, \trigblock{} has higher 
computational cost than \t{trigmv}.
The increase in matrix--vector products is most pronounced for 
normal matrices (\texttt{bcspwr10} and
\texttt{L2}).
For the matrix \texttt{bcspwr10} we find
$s=7$, $mv=618$, and $mvd=44$ (matrix--vector products performed in the 
preprocessing stage, in \cref{cf:ms}) for \trigmv. 
On the other hand, for \trigblock{} we find 
$s=10$, $mv=696\cdot 2=1392$, and $mvd=328\cdot 2=656$.
This means that the preprocessing stage is more expensive as the block 
matrix is nonnormal and more $\alpha_p$ values need to be computed. 
We can also see that we need more scaling steps as we miss the opportunity to 
reduce the norm. 
In total this sums up to more than twice the number of matrix--vector
products. 

The results of the experiment for the hyperbolic functions can be seen in
\cref{tab:large matrices_trigh}.
Again \trighmv{} almost always needs fewer matrix--vector products than the 
other methods 
where this time \trighexpmv{} is the closest competitor and 
\trighblock{} has a higher computational effort.
Even in the cases where \trighexpmv{} needs the same number of 
matrix--vector products or slightly fewer,
\trighmv{} is still clearly faster.
This is due to the fact that \trigmv\ employs level-3 BLAS. 


Comparing the runtime of \trigmv{} and \trighmv{} with the \dense{} algorithms
we can see that we potentially save a great deal of computation time.
The \texttt{triw} and \texttt{triu} 
matrices are the only cases where there is not a speedup of
at least a factor of 10.
For the \texttt{triw} matrix, and to a lesser extent for the 
\texttt{triu} matrix, the $\alpha_p$ values, 
which help deal with the nonnormality of the matrix, 
decay very slowly, and this hinders 
the performance of the algorithms. 
Nevertheless, in all the other cases we can see a clear speed advantage,
most significantly for \texttt{bcspwr10} where we have a speedup by a factor
6190.

\end{eg}

\begin{eg}[Schr\"odinger equation]\label{eg:schroedinger}
In this example we solve an evolution equation.
We consider the 3D Schr\"odinger equation with harmonic potential
\begin{align}\label{eq:schroedinger}
\partial_t u = \frac{\iu}{2}\left(\Delta -
\frac12 \left(x^2 + y^2 + z^2\right)\right)u.
\end{align}
We use a finite difference discretization in space with $N^3$ points on the
domain $\Omega=[0,1]^3$ and
as initial value we use the discretization of
$4096x^2(1-x)^2y^2(1-y)^2z^2(1-z)^2$.
We obtain a discretization matrix $\iu A$ of size $27000\times 27000$,
where $A$ is symmetric with all eigenvalues on the negative real axis. 
We deliberately keep $\iu$ separate and as a result
the solution of \cref{eq:schroedinger} can be interpreted as
\begin{align*}
u(t)=\eu^{\iu tA}u_0=\cos(t A)u_0 + \iu \sin(t A)u_0.
\end{align*}

\begin{table}\centering\captionsetup{position=top}\footnotesize
%
%
%
\caption{\label{tab:schroedinger}Results for the solution of the
 Schr\"odinger equation with $N=30$.
 We show the number of matrix--vector products performed,
 the relative error in the $1$-norm, and the CPU time.
 }
\begin{tabular}{l|c|c|c|c|c|l}
    & \multicolumn{3}{c|}{$\tol=\tols$} & \multicolumn{3}{c}{$\tol=\told$}\\
               & $mv$  & rel.~err & Time    & $mv$  & rel.~err & Time   \\\hline
 \trigmv{}     & 11034 & 1.3e-7 & 3.9 & 15846 & 2.7e-11 & 5.6 \\
 \trigexpmv{}  & 21952 & 1.3e-7 & 6.2 & 31516 & 2.7e-11 & 8.8 \\
 \trigblock{}  & 15883 & 5.2e-8 & 7.1 & 32023 & 1.1e-11 & 1.4e1\\
 \expleja{}    & 11180 & 8.0e-9 & 4.3 & 17348 & 1.5e-11 & 6.6 \\
 \texttt{dense}&   -   & -      & -   & -     & -       & 2.2e4
\end{tabular}


\end{table}

\Cref{tab:schroedinger} reports the results for the tolerances
$\tols$ and $\told$,
for our new algorithm \trigmv{}, \trigexpmv{}, \trigblock{}, and
\expleja{} (the method from \cite{ckor15} called in the same fashion as 
\trigexpmv{}).
The table shows the number of matrix--vector products performed,
the relative error, and the CPU time in seconds.
We see that the four methods achieve roughly the same accuracy.
We also see that \trigmv{} requires
significantly fewer matrix--vector products than \trigexpmv{} and
\trigblock{}. 
On the other hand, even though \expleja{} is a close competitor in terms of
matrix--vector products performed  the overall CPU time is higher 
than for \trigmv{}.
This is due to the fact that \trigmv{} is avoiding complex arithmetic and 
employs level-3 BLAS.
Also note that \trigmv{} needs less storage than \expleja{} as for the latter
the matrix needs to be complex.
Again we can see that the dense method needs roughly 1000 times longer for
the computation than the other algorithms.
\end{eg}

\section{Concluding remarks}\label{sec.conc}

We have developed the first \alg\ for computing the actions of the matrix
functions $\cos A$, $\sin A$, $\cosh A$, and $\sinh A$.
Our new \alg, \cref{alg:basic_alg},
can evaluate the individual actions or the actions of any of the
functions simultaneously.
The \alg\ builds on the framework of the $\eu^Ab$ \alg\ \expmv\ of
Al-Mohy and Higham \cite{alhi09a}, 
inheriting its backward stability
with respect to truncation errors,
its exclusive use of matrix--vector products
(or matrix--matrix products in our modification),
and its features for countering the effects of nonnormality.
For real $A$, $\cos(A)b$ and $\sin(A)b$ are computed entirely
in real arithmetic.
As a result of these features and its careful reuse of information,
\cref{alg:basic_alg} is more efficient than alternatives
that make multiple calls to \expmv,
as our experiments demonstrate.


Our MATLAB codes are available at \url{https://bitbucket.org/kandolfp/trigmv}

\section*{Acknowledgement}
The computational results presented have been achieved (in part)
using the HPC infrastructure LEO of the University of Innsbruck.
We thank Awad H.~Al-Mohy for his comments on an early version of 
the manuscript. 
We thank the referees for their constructive remarks which helped us to improve 
the presentation of this paper.


\newpage
\bibliographystyle{siamplain}
\bibliography{strings,trigmv,njhigham}

\end{document}